\documentclass[12pt]{amsart}

\usepackage{amssymb,latexsym}

\usepackage{enumerate}

\makeatletter

\@namedef{subjclassname@2010}{

  \textup{2010} Mathematics Subject Classification}

\makeatother
\newtheorem{thm}{Theorem}[section]
\newtheorem{Con}[thm]{Conjecture}

\newtheorem{cor}[thm]{Corollary}
\newtheorem{lem}[thm]{Lemma}
\newtheorem{pro}[thm]{Proposition}
\theoremstyle{definition}

\newtheorem{rem}[thm]{Remark}

\numberwithin{equation}{section}

\newcommand{\pr}{\mathbb{P}}
\newcommand{\ex}{\mathbb{E}}
\newcommand{\re}{\textup{Re}}
\newcommand{\im}{\textup{Im}}

\begin{document}

\baselineskip=17pt

\title{A bias in  Mertens' product formula}

\author[Youness Lamzouri]{Youness Lamzouri}

\address{Department of Mathematics and Statistics,
York University,
4700 Keele Street,
Toronto, ON,
M3J1P3
Canada}

\email{lamzouri@mathstat.yorku.ca}

\date{}

\begin{abstract}  Rosser and Schoenfeld remarked that the product $\prod_{p\leq x}(1-1/p)^{-1}$ exceeds $e^{\gamma} \log x$ for all $2\leq x\leq 10^8$, and raised the question whether the difference changes sign infinitely often. This was confirmed in a recent paper of Diamond and Pintz. In this paper, we show (under certain hypotheses) that there is a strong bias in the race between the product $\prod_{p\leq x}(1-1/p)^{-1}$ and $e^{\gamma}\log x$ which explains the computations of Rosser and Schoenfeld. 

\end{abstract}


\thanks{The author is partially supported by a Discovery Grant from the Natural Sciences and Engineering Research Council of Canada.}

\maketitle

\section{Introduction}

In 1874 Mertens proved three remarkable results on the distribution of prime numbers. His third theorem asserts that 
$$ \prod_{p\leq x}\left(1-\frac1p\right)^{-1}\sim e^{\gamma}\log x, \text{ as } x\to\infty,$$
where $\gamma$ is the Euler-Mascheroni constant. Rosser and Schoenfeld \cite{RoSc} noticed that for all $2\leq x\leq 10^8$,  we have
\begin{equation}\label{ineq}
\prod_{p\leq x}\left(1-\frac1p\right)^{-1}> e^{\gamma}\log x,
\end{equation}
and suggested that ``perhaps'' one can prove that the difference changes sign for arbitrarily large $x$, in analogy to Littlewood's classical result for $\pi(x)-\text{Li}(x)$. Recently, Diamond and Pintz \cite{DP} investigated this question and confirmed Rosser and Schoenfeld prediction. More precisely, they established that the quantity
\begin{equation}\label{diff}
 \sqrt{x}\left(\prod_{p\leq x}\left(1-\frac1p\right)^{-1}- e^{\gamma}\log x\right)
\end{equation}
attains arbitrarily large positive and negative values as $x\to \infty$. Let $\mathcal{M}$ be the set of real numbers $x\geq 2$ such that 
$$ 
\prod_{p\leq x}\left(1-\frac{1}{p}\right)^{-1}>e^{\gamma}\log x.
$$
Then, Diamond and Pintz result asserts that both $\mathcal{M}$ and its complement are unbounded. Assuming the Riemann hypothesis RH we strengthen this result by proving that both $\mathcal{M}$ and its complement have positive \emph{lower logarithmic densities}. Recall that for a set $S\subset [0, \infty)$, the upper and lower logarithmic densities of $S$ are defined respectively by
$$\overline{\delta}(S)= \limsup_{x\to\infty}\frac{1}{\log x}\int_{t\in S\cap [2,x]}\frac{dt}{t}, \text{ and } \underline{\delta}(S)= \liminf_{x\to\infty}\frac{1}{\log x}\int_{t\in S\cap [2,x]}\frac{dt}{t}.$$
If $\overline{\delta}(S)=\underline{\delta}(S)=\delta(S)$ we say that $\delta(S)$ is the logarithmic density of $S$. 
We prove 
\begin{thm}\label{Positive}
Assume RH. Then $\underline{\delta}(\mathcal{M})>0$ and $\overline{\delta}(\mathcal{M})<1$. 
\end{thm}

\begin{rem} Note that the assumption of the Riemann hypothesis in Theorem \ref{Positive} is ``necessary'' in a certain sense. Indeed, using the work of the author with Ford and Konyagin \cite{FKL} one can show that the existence of certain configurations of zeros of the Riemann zeta function $\zeta(s)$ off the critical line implies that $\delta(\mathcal{M})=0$. 
\end{rem}

A natural question to ask is which of the quantities $\prod_{p\leq x}(1-1/p)^{-1}$ and $e^{\gamma}\log x$ is larger most of the time? Although Diamond and Pintz result shows that both take the lead for arbitrarily large $x$, the computations of Rosser and Schoenfeld seem to suggest that the product $\prod_{p\leq x}(1-1/p)^{-1}$ predominates.   Assuming the Riemann hypothesis together with a further assumption we explain this phenomenon, by showing that the difference $\prod_{p\leq x}\left(1-1/p\right)^{-1}- e^{\gamma}\log x$ has a strong tendency to be positive. The hypothesis we assume is the Linear Independence hypothesis LI, which is the assumption that the positive imaginary parts of the non-trivial zeros of  $\zeta(s)$ are linearly independent over $\mathbb{Q}$.

\begin{thm}\label{Bias} Assume RH and LI. Then the set $\mathcal{M}$ has logarithmic density $$\delta(\mathcal{M})= 0.99999973...$$
\end{thm}

Rubinstein and Sarnak \cite{RuSa} have previously used the hypotheses RH and LI (and their generalizations for Dirichlet $L$-functions) to study several prime number races, including the race between $\pi(x)$ and $\text{Li(x)}$ and the Shanks-R\'enyi race between $\pi(x; q, a)$ and $\pi(x; q, b)$ for different arithmetic progressions $a, b\pmod q$, where $\pi(x;q,a)$ is the number of primes $p\leq x$ such that $p\equiv a \pmod q$. In particular, they explained and quantified Chebyshev's observation in 1853 that primes congruent to $3\pmod 4$ predominate over those congruent to $1\pmod 4$. In general, if $a$ is a non-square modulo $q$
and $b$ is a square modulo $q$ then $\pi(x; q, a)$ has a strong tendency to be larger than
$\pi(x; q, b)$, a phenomenon which has become known as ``Chebyshev's bias''. For more on the history of this subject as well as recent developments, the reader is invited to consult the expository papers of Granville and Martin \cite{GM} and Martin and Scarfy \cite{MS}.

\begin{rem} Under RH and LI, it turns out that $\delta(\mathcal{M})=1-\delta_0$ where $\delta_0$ is the logarithmic density of the set of real numbers $x\geq 2$ for which $\pi(x)>\text{Li}(x)$. We shall explain why this is the case in Section 4 below. 
\end{rem}

Curiously, a similar phenomenon to Chebyshev's bias for primes in arithmetic progressions does not appear when we consider the analogous problem of comparing the Mertens products
\begin{equation}\label{Mertens}
\prod_{\substack{p\leq x\\ p\equiv a \bmod q}}\left(1-\frac1p\right)^{-1}, 
\end{equation}
 for different arithmetic progressions $a\pmod q$. Indeed,  Williams \cite{Wi} proved that for any $(a,q)=1$, there exists a constant $c(a,q)>0$ such that 
$$ \prod_{\substack{p\leq x\\ p\equiv a \bmod q}}\left(1-\frac1p\right)^{-1} \sim c(a, q) (\log x)^{1/\phi(q)}, \text{ as } x\to \infty.$$
Thus, if $c(a,q)>c(b,q)$ then the residue class $a\pmod q$ is guaranteed to win the Mertens product race as soon as $x$ exceeds a certain number that depends only on $a, b$ and $q$. Languasco and Zaccagnini \cite{LZ} computed many of the constants $c(a,q)$ and showed that for example $c(3,4)>c(1,4)$, but $c(2,7)>c(3,7)$ although $2$ is a quadratic residue and $3$ is a quadratic non-residue modulo $7$.  The difference from Chebyshev's bias probably lies in the fact that the product \eqref{Mertens} is heavily affected by the  small primes $p\equiv a \pmod q$ due to the factor $1/p$. Therefore, if an arithmetic progression contains many small primes, then it has a better chance to win in the Mertens product race. 


Concerning the size of the oscillations of the difference \eqref{diff}, Diamond and Pintz \cite{DP} proved that
$$  
\sqrt{x}\left(\prod_{p\leq x}\left(1-\frac1p\right)^{-1}- e^{\gamma}\log x\right)= \Omega_{\pm} \left(\log\log\log x\right).
$$ 
Montgomery \cite{Mont} used probabilistic arguments to conjecture the maximal size of $\pi(x)-\text{Li}(x)$. Following his approach we make the following conjecture

\begin{Con}\label{MaxMin}
As $x\to \infty$ we have 
$$ \limsup_{x\to\infty} \frac{\sqrt{x}}{(\log\log\log x)^2}\left(\prod_{p\leq x}\left(1-\frac1p\right)^{-1}- e^{\gamma}\log x\right)=\frac{e^{\gamma}}{2\pi}, $$
and 
$$ \liminf_{x\to\infty} \frac{\sqrt{x}}{(\log\log\log x)^2}\left(\prod_{p\leq x}\left(1-\frac1p\right)^{-1}- e^{\gamma}\log x\right)=-\frac{e^{\gamma}}{2\pi}.$$
\end{Con}

\section{An explicit formula for the remainder and the origin of the bias}
Let 
$$
E_M(x):= \sqrt{x}(\log x)\left(\log \left(\prod_{p\leq x}\left(1-\frac1p\right)^{-1}\right) -\log\log x -\gamma\right).
$$
Then, observe that 

\begin{equation}\label{equiv}
\prod_{p\leq x}\left(1-\frac{1}{p}\right)^{-1}>e^{\gamma}\log x \ \text{ if and only if } \ E_M(x)>0.
\end{equation}

The key ingredient in the proofs of Theorems  \ref{Positive} and  \ref{Bias} is the following unconditional explicit formula for $E_M(x)$ is terms of the non-trivial zeros of $\zeta(s)$.
\begin{pro}\label{MainExplicit}
For any real numbers $x, T \geq 5$  we have
$$
E_M(x)=1
+ \sum_{|\im(\rho)|<T} \frac{x^{\rho-1/2}}{\rho-1}+ O\left(\frac{1}{\log x}\sum_{|\im(\rho)|<T}\frac{x^{\re(\rho)-1/2}}{\im(\rho)^2}+\frac{\sqrt{x}(\log (xT))^2}{T}+\frac{1}{\log x}\right),
$$
where $\rho$ runs over the non-trivial zeros of $\zeta(s).$
\end{pro}
From this formula one can deduce that the source of the bias is the constant $1$  which comes from the contribution of the squares of primes (see Lemma \ref{Approx1} below). Indeed, if we assume the Riemann hypothesis,  we get the following corollary. 

\begin{cor}\label{ExplicitRH}
Assume the Riemann hypothesis, and let $1/2+i\gamma_n$ runs over the non-trivial zeros of $\zeta(s)$.  Then, for any real numbers $x, T\geq 5$ we have 
\begin{equation}\label{ExplicitFormula}
 E_M(x)=  1 + 2 \re \sum_{0<\gamma_n<T} \frac{x^{i\gamma_n}}{-1/2+i \gamma_n}+ O\left(\frac{\sqrt{x}(\log (xT))^2}{T}+\frac{1}{\log x}\right).
\end{equation}
\end{cor}
\begin{proof} This follows from Proposition \ref{MainExplicit}  along with the fact that 
\begin{equation}\label{BoundZeros}
\sum_{|\gamma_n|<T}\frac{1}{\gamma_n^2} \ll 1
\end{equation}
by the Riemann-von Mangoldt formula.

\end{proof}

In order to prove Proposition \ref{MainExplicit} we first need the following lemmas. 

\begin{lem}\label{Approx1}
For any real number $x\geq 2$ we have 
$$\log \left(\prod_{p\leq x}\left(1-\frac1p\right)^{-1}\right)=\sum_{n\leq x}\frac{\Lambda(n)}{n\log n} +\frac{1}{\sqrt{x}\log x}+O\left(\frac{1}{\sqrt{x}(\log x)^2}\right).$$
\end{lem}
\begin{proof}
We have 
\begin{equation}\label{Approx2}
\begin{aligned}\log \left(\prod_{p\leq x}\left(1-\frac1p\right)^{-1}\right) = \sum_{p\leq x} \sum_{k=1}^{\infty} \frac{1}{kp^k}&= \sum_{n\leq x}\frac{\Lambda(n)}{n\log n} +\sum_{\substack{ k\geq 2\\ x^{1/k}<p\leq x}}\frac{1}{kp^k}\\
&= \sum_{n\leq x}\frac{\Lambda(n)}{n\log n}  +\sum_{\sqrt{x}<p\leq x} \frac{1}{2p^2}+ O\left(x^{-2/3}\right).
\end{aligned}
\end{equation}
Furthermore, by the prime number theorem, we have 
$$ \sum_{\sqrt{x}<p\leq x}\frac{1}{p^2}=\int_{\sqrt{x}}^x\frac{d\pi(t)}{t^2}= \int_{\sqrt{x}}^x\frac{dt}{t^2\log t} +O\left(x^{-1/2}e^{\sqrt{\log x}}\right).$$
We use the change of variable $u=\log t-(\log x)/2$ to deduce that
\begin{align*}
\int_{\sqrt{x}}^x\frac{dt}{t^2\log t}&= \frac{1}{\sqrt{x}}\int_0^{(\log x)/2}\frac{2e^{-u}}{2u+\log x}du\\
&= \frac{2}{\sqrt{x}\log x}\int_0^{(\log x)/2}e^{-u}du + O\left(\frac{1}{\sqrt{x}(\log x)^2}\right)\\
&= \frac{2}{\sqrt{x}\log x}+ O\left(\frac{1}{\sqrt{x}(\log x)^2}\right).
\end{align*}
Inserting this estimate in \eqref{Approx2} completes the proof.
\end{proof}

\begin{lem}\label{Perron}
For any $\alpha>1$ and $x, T\geq 5$ we have 
\begin{align*}
\sum_{n\leq x} &\frac{\Lambda(n)}{n^{\alpha}} = -\frac{\zeta'}{\zeta}(\alpha)+\frac{x^{1-\alpha}}{1-\alpha} -\sum_{|\im(\rho)|\leq T} \frac{x^{\rho-\alpha}}{\rho-\alpha}\\
&+ O\left(x^{-\alpha}\log x + \frac{x^{1-\alpha}}{T}\left(4^{\alpha}+(\log x)^2+\frac{(\log T)^2}{\log x} \right)+\frac{1}{T}\sum_{n=1}^{\infty}\frac{\Lambda(n)}{n^{\alpha+1/\log x}} \right).\\
\end{align*}
\end{lem}

\begin{proof}  Since there are $O(\log T)$ non-trivial zeros of $\zeta(s)$ with ordinate in $[T, T+1]$, then there exists a point $T_0\in [T, T+1]$ which is at a distance $\gg 1/\log T$ from the nearest zero of $\zeta(s)$.  Let $c=1/\log x$ and consider the integral
\begin{equation}\label{Mellin}
\frac{1}{2\pi i}\int_{c-iT_0}^{c+iT_0}\left(-\frac{\zeta'}{\zeta}(\alpha+s)\right)\frac{x^s}{s}ds.
\end{equation}
First,  by Perron's formula the integral above equals 
$$\sum_{n\leq x} \frac{\Lambda(n)}{n^{\alpha}}+ O\left(\sum_{n=1}^{\infty}\frac{\Lambda(n)}{n^{\alpha+c}}\min\left(1, \frac{1}{T_0|\log(x/n)|}\right)\right).$$
To bound the error term of this last estimate, we first handle the terms  $n\leq x/2$ and $n\geq 2x$. These satisfy $|\log(x/n)|\geq \log 2$, and hence their contribution is 
$$\ll \frac{1}{T} \sum_{n=1}^{\infty}\frac{\Lambda(n)}{n^{\alpha+c}}.$$
Now for $x/2<n<2x$, we let $r=n-x$. The terms with $|r|\leq 1$ contribute 
$\ll x^{-\alpha} \log x$. Furthermore, if $|r|\geq 1$ we use the bound $|\log(x/n)|\gg |r|/x$. Hence,  the contribution of these terms is 
$$\ll \frac{x^{1-\alpha}\log x}{T}\sum_{1\leq |r|\leq x}\frac{1}{|r|}\ll \frac{x^{1-\alpha}(\log x)^2}{T}.$$
Therefore, we deduce that the integral \eqref{Mellin} equals
\begin{equation}\label{FirstEvaluation}
 \sum_{n\leq x} \frac{\Lambda(n)}{n^{\alpha}} +O\left(x^{-\alpha}\log x+ \frac{x^{1-\alpha}(\log x)^2}{T}+\frac{1}{T}\sum_{n=1}^{\infty}\frac{\Lambda(n)}{n^{\alpha+1/\log x}} \right).
 \end{equation}
We now move the contour of integration in \eqref{Mellin} to the line $\text{Re}(s)=-U$ where $U>0$ is large and $U\neq 2n+\alpha$ for any $n\in \mathbb{N}$. We encounter simple poles at $0, 1-\alpha$ and $z-\alpha$ for every zero $z$ of $\zeta(s)$ with $|\im(z)|\leq T_0$ and $\text{Re}(z)>-U$. Evaluating the residues there, we find that our integral equals
\begin{equation}\label{SecondEvaluation}
 -\frac{\zeta'}{\zeta}(\alpha)+\frac{x^{1-\alpha}}{1-\alpha} -\sum_{|\im(\rho)|\leq T_0} \frac{x^{\rho-\alpha}}{\rho-\alpha} +\sum_{n\leq (U-\alpha)/2}\frac{x^{-2n-\alpha}}{2n+\alpha} +I ,
 \end{equation}
where 
$$
I
=  \frac{1}{2\pi i}\left(\int_{c-iT_0}^{-U-iT_0}+\int_{-U-iT_0}^{-U+iT_0}+\int_{-U+iT_0}^{c+iT_0}\right) \left(-\frac{\zeta'}{\zeta}(\alpha+s)\right)\frac{x^s}{s}ds.
$$
To bound the first and third integrals, we first note that for all $\re(s)=\sigma\geq 1-\alpha+c$ we have 
$$ \left|-\frac{\zeta'}{\zeta}(\alpha+s)\right| \leq \sum_{n=1}^{\infty} \frac{\Lambda(n)}{n^{\alpha+\sigma}}.$$
On the other hand, if  $\re(s)=\sigma\leq 1-\alpha+c$ we use the following estimate for $\zeta'/\zeta(s)$ (see for example equation (4) of  Chapter 15 of  Davenport \cite{Da})
\begin{equation}\label{Dav}
\frac{\zeta'}{\zeta}(\sigma+it)=\sum_{|t-\im(\rho)|\leq 1}\frac{1}{\sigma+it-\rho}+O(\log (|t|+2)).
\end{equation}
Then, using our assumption on $T_0$ we obtain
\begin{align*}
\frac{1}{2\pi i}\int_{c-iT_0}^{-U-iT_0} \left(-\frac{\zeta'}{\zeta}(\alpha+s)\right)\frac{x^s}{s}ds
&\ll  \frac{(\log T)^2}{T}\int_{-U}^{1-\alpha+c}x^{\sigma}d\sigma +\frac{1}{T} \sum_{n=1}^{\infty} \frac{\Lambda (n)}{n^{\alpha}} \int_{1-\alpha+c}^{c} \left(\frac{x}{n}\right)^{\sigma}d\sigma\\
&\ll \frac{(\log T)^2x^{1-\alpha}}{T\log x}+\frac{1}{T} \sum_{n=1}^{\infty} \frac{\Lambda (n)}{n^{\alpha}} \int_{1-\alpha+c}^{c} \left(\frac{x}{n}\right)^{\sigma}d\sigma.
\end{align*}
To bound the second term in the right hand side of this estimate, we split the sum according to $n\leq x/2$, $x/2<n<2x$, and $n\geq 2x$. For the first and third terms, we use that
$$ \int_{1-\alpha+c}^{c} \left(\frac{x}{n}\right)^{\sigma}d\sigma \ll \frac{(x/n)^c+(x/n)^{1-\alpha+c}}{|\log(x/n)|}\ll \frac{1}{n^c}+ \frac{x^{1-\alpha}}{n^{1-\alpha+c}},$$
while for the middle terms, we simply bound the integrand trivially, to obtain
$$  \int_{1-\alpha+c}^{c} \left(\frac{x}{n}\right)^{\sigma}d\sigma\ll (\alpha-1) 2^{\alpha}.$$
Hence, we derive 
\begin{align*}
 \sum_{n=1}^{\infty}\frac{\Lambda (n)}{n^{\alpha}} \int_{1-\alpha+c}^{c} \left(\frac{x}{n}\right)^{\sigma}d\sigma
 &\ll (\alpha-1) 2^{\alpha}\sum_{x/2<n}\frac{\Lambda(n)}{n^{\alpha}}+x^{1-\alpha}\sum_{n=1}^{\infty}\frac{\Lambda(n)}{n^{1+c}} 
 + \sum_{n=1}^{\infty} \frac{\Lambda(n)}{n^{\alpha+c}}\\
 &\ll  x^{1-\alpha}(4^{\alpha} +\log x)+ \sum_{n=1}^{\infty} \frac{\Lambda(n)}{n^{\alpha+c}},
\end{align*}
by the prime number theorem.
Therefore, we obtain
$$ \frac{1}{2\pi i}\int_{c-iT_0}^{-U-iT_0} \left(-\frac{\zeta'}{\zeta}(\alpha+s)\right)\frac{x^s}{s}ds 
\ll \frac{x^{1-\alpha}}{T} \left(4^{\alpha}+\log x+ \frac{(\log T)^2}{\log x}\right) + \frac{1}{T} \sum_{n=1}^{\infty} \frac{\Lambda (n)}{n^{\alpha+1/\log x}}.
$$
A similar bound holds for $\frac{1}{2\pi i}\int_{-U+iT_0}^{c+iT_0} -\zeta'/\zeta(\alpha+s)\frac{x^s}{s}ds$. Moreover, by \eqref{Dav} we obtain
$$ \frac{1}{2\pi i}\int_{-U-iT_0}^{-U+iT_0} \left(-\frac{\zeta'}{\zeta}(\alpha+s)\right)\frac{x^s}{s}ds\ll \frac{(\log T)^2}{x^U}.
$$
Combining these estimates and letting $U\to \infty$, we deduce that
\begin{equation}\label{LeftContour}
I\ll  \frac{x^{1-\alpha}}{T}\left(4^{\alpha}+\log x+ \frac{(\log T)^2}{\log x}\right) + \frac{1}{T} \sum_{n=1}^{\infty} \frac{\Lambda (n)}{n^{\alpha+1/\log x}}.
\end{equation}
Furthermore, note that
$$ 
\sum_{n\leq (U-\alpha)/2}\frac{x^{-2n-\alpha}}{2n+\alpha}\leq \sum_{n=1}^{\infty}\frac{x^{-2n-\alpha}}{2n+\alpha}\ll x^{-2-\alpha},
$$
and 
$$
\sum_{T\leq |\im(\rho)|\leq T_0} \frac{x^{\rho-\alpha}}{\rho-\alpha}\ll \frac{x^{1-\alpha}\log T}{T}.
$$
Inserting these two estimates in \eqref{SecondEvaluation} and using  \eqref{FirstEvaluation}  and \eqref{LeftContour} completes the proof.
\end{proof}

\begin{lem}\label{Gamma} For any $x\geq 2$ we have 
$$
\lim_{\sigma\to 1^{+}}\left(\log\zeta(\sigma)+\int_{\sigma}^{\infty} \frac{x^{1-\alpha}}{1-\alpha}d\alpha\right)=\log\log x+\gamma.
$$
\end{lem}

\begin{proof}
Let $\sigma>1$ and $y_0=(\sigma-1)\log x$.  Using the change of variables $y=(\alpha-1)\log x$ we obtain
$$ \int_{\sigma}^{\infty} \frac{x^{1-\alpha}}{1-\alpha}d\alpha=-\int_{y_0}^{\infty}\frac{e^{-y}}{y}dy= \log\log x+\log(\sigma-1)+ \int_{y_0}^{1}\frac{1-e^{-y}}{y} dy-\int_1^{\infty}\frac{e^{-y}}{y}dy.$$ 
Moreover, since $\log\zeta(\sigma)=-\log(\sigma-1)+O(\sigma-1),$  we derive
$$\lim_{\sigma\to 1^{+}}\left(\log\zeta(\sigma)+\int_{\sigma}^{\infty} \frac{x^{1-\alpha}}{1-\alpha}d\alpha\right)= \log\log x+ \int_{0}^{1}\frac{1-e^{-y}}{y} dy-\int_1^{\infty}\frac{e^{-y}}{y}dy.$$
Finally, note that (see for example Section 5.1 of Abramowitz-Stegun \cite{AbSt})
$$\int_{0}^{1}\frac{1-e^{-y}}{y} dy-\int_1^{\infty}\frac{e^{-y}}{y}dy=\gamma.$$
\end{proof}

We are now ready to prove the explicit formula for $E_M(x)$. 
\begin{proof}[Proof of Proposition \ref{MainExplicit}]
Let $\sigma> 1$ be fixed. Then by Lemma \ref{Perron} we have 
\begin{equation}\label{integration}
\begin{aligned}
 \sum_{n\leq x}\frac{\Lambda(n)}{n^{\sigma}\log n}&=\int_{\sigma}^{\infty} \sum_{n\leq x}\frac{\Lambda(n)}{n^{\alpha}}d\alpha\\
&= \log\zeta(\sigma)+\int_{\sigma}^{\infty} \frac{x^{1-\alpha}}{1-\alpha}d\alpha - \sum_{|\im(\rho)|\leq T}\int_{\sigma}^{\infty}\frac{x^{\rho-\alpha}}{\rho-\alpha} d\alpha+ E_1,
\end{aligned}
\end{equation}
where 
$$E_1\ll \frac{1}{T}\left(\log x+\frac{(\log T)^2}{(\log x)^2}\right)+ \frac{1}{x}+\frac{1}{T}\sum_{n=1}^{\infty}\frac{\Lambda(n)}{n^{1+1/\log x}\log n}\ll \frac{1}{T}\left(\log x+\frac{(\log T)^2}{(\log x)^2}\right)+ \frac{1}{x}. $$
Taking the limit as $\sigma\to 1^+$ of both sides of \eqref{integration} and using Lemma \ref{Gamma} we deduce that
\begin{equation}\label{IntLambda}
 \sum_{n\leq x}\frac{\Lambda(n)}{n\log n}= \log\log x+\gamma- \sum_{|\im(\rho)|\leq T}x^{\rho}\int_{1}^{\infty}\frac{x^{-\alpha}}{\rho-\alpha}d\alpha +O\left(\frac{\log x}{T} +\frac{(\log T)^2}{T(\log x)^2}+ \frac{1}{x}\right).
\end{equation}
To evaluate the integral in the right hand side of this estimate, we make the change of variable $u=(\alpha-1)\log x$ to obtain
$$ \int_{1}^{\infty}\frac{x^{-\alpha}}{\rho-\alpha}d\alpha= \frac{1}{x}\int_0^{\infty} \frac{e^{-u}}{(\rho-1)\log x-u} du.$$
Note that $|(\rho-1)\log x-u|\geq |\im(\rho)|\log x$ for all $u\in \mathbb{R}$, and hence
$$ \frac{1}{(\rho-1)\log x-u}=\frac{1}{(\rho-1)\log x}+O\left(\frac{u}{(\im(\rho)\log x)^2}\right).$$
Therefore, we obtain 
$$ \int_{1}^{\infty}\frac{x^{-\alpha}}{\rho-\alpha}d\alpha= \frac{1}{x(\log x)(\rho-1)} +O\left(\frac{1}{x(\log x)^2(\im(\rho))^2}\right).$$
Inserting this estimate in \eqref{IntLambda} and appealing to Lemma \ref{Approx1} completes the proof.
\end{proof}

\section{Proof of Theorem \ref{Positive}}

In this section we shall use the explicit formula \eqref{ExplicitFormula} along with the work of Rubinstein and Sarnak \cite{RuSa} to prove that both $\mathcal{M}$ and its complement have positive lower logarithmic densities. 

\begin{proof}[Proof of Theorem \ref{Positive}]
Let $Y$ be large and $x=e^Y$. First, by making the change of variable $y=\log t$ we deduce that 
\begin{equation}\label{change}
\begin{aligned}
\frac{1}{\log x}\int_{t\in \mathcal{M}\cap [2, x]}\frac{dt}{t}&= \frac1Y\text{meas}\left\{ \log 2\leq y\leq Y: e^y\in \mathcal{M}\right\}\\
&= \frac1Y\text{meas}\left\{ \log 2\leq y\leq Y: E_M(e^y)>0\right\}.
\end{aligned}
\end{equation}
By Corollary \ref{ExplicitRH} and equation \eqref{BoundZeros} we have for all $T\geq 5$ and $y\geq 2$
$$E_M(e^y)= 2\sum_{0<\gamma_n<T}\frac{\sin(\gamma_n y)}{\gamma_n}+ O\left(1+\frac{(y+\log T)^2 e^{y/2}}{T}\right).$$
Therefore, we deduce that if $Y$ is large enough, then there exists a suitably large constant $A>0$ such that 
\begin{equation}\label{Order}
 2\left(\sum_{0<\gamma_n<e^Y}\frac{\sin(\gamma_n y)}{\gamma_n}-A\right)< E_M(e^y) < 2\left(\sum_{0<\gamma_n<e^Y}\frac{\sin(\gamma_n y)}{\gamma_n} +A\right), 
\end{equation}
for all $2\leq y\leq Y$. 
  
Based on the approach of Littlewood \cite{Li}, Rubinstein and Sarnak proved in Section 2.2 of \cite{RuSa} that 
for all $\lambda \gg  1$ we have 
\begin{equation}\label{Right}
 \frac1Y\text{meas}\left\{ 2\leq y\leq Y: \sum_{0<\gamma_n< e^Y}\frac{\sin(\gamma_n y)}{\gamma_n}>\lambda\right\} \geq c_1 \exp\big(-\exp(-c_2\lambda)\big), 
\end{equation}
and 
\begin{equation}\label{Left}
 \frac1Y\text{meas}\left\{2\leq y\leq Y: \sum_{0<\gamma_n< e^Y}\frac{\sin(\gamma_n y)}{\gamma_n}<-\lambda\right\} \geq c_1\exp\big(-\exp(-c_2\lambda)\big), 
\end{equation}
for some absolute positive constants $c_1, c_2$, if $Y$ is large enough. Therefore,  combining equations \eqref{change}, \eqref{Order} and \eqref{Right} we obtain 
\begin{align*}
\frac{1}{\log x}\int_{t\in \mathcal{M}\cap [2, x]}\frac{dt}{t}
&\geq \frac1Y\text{meas}\left\{ 2\leq y\leq Y: \sum_{0<\gamma_n< e^Y}\frac{\sin(\gamma_n y)}{\gamma_n}>A \right\} \\
&\geq \frac{c_1}{2}\exp\big(-\exp(-c_2A)\big),
\end{align*}
if $Y$ is large enough.
Thus, we deduce that $\underline{\delta}(\mathcal{M})\geq \frac{c_1}{2} \exp\big(-\exp(-c_2A)\big)>0$. Similarly, by \eqref{Left} we have 
\begin{align*}
\frac{1}{\log x}\int_{t\in \mathcal{M}\cap [2, x]}\frac{dt}{t}
&\leq \frac1Y\text{meas}\left\{ 2\leq y\leq Y: \sum_{0<\gamma_n< e^Y}\frac{\sin(\gamma_n y)}{\gamma_n}>-A \right\} +O\left(\frac{1}{Y}\right)\\
&\leq 1-  \frac{c_1}{2}\exp\big(-\exp(-c_2A)\big).
\end{align*}
Hence, we get $\overline{\delta}(\mathcal{M})\leq 1-\frac{c_1}{2} \exp\big(-\exp(-c_2A)\big) <1,$ as desired.

\end{proof} 

\section{A limiting distribution for $E_M(x)$ and proof of Theorem \ref{Bias}}

Assuming the Riemann hypothesis and using the explicit formula \eqref{ExplicitFormula} we deduce that the quantity $E_M(x)$ has \textit{a logarithmic limiting distribution}. This follows from the fact that $E_M(e^y)$ is   a $B^2$\emph{-almost periodic} function. More precisely, we have 
\begin{pro}\label{LimitDistrib}
Assume RH.  Then there exists a probability measure $\mu_M$ on $\mathbb{R}$ such that
$$ \lim_{x\to\infty} \frac{1}{\log x}\int_2^x f\big(E_M(t)\big)\frac{dt}{t}= \int_{-\infty}^{\infty} f(t)d\mu_M,$$
for all bounded continuous functions on $\mathbb{R}$. 
\end{pro} 
\begin{proof}
This follows from the analysis in Rubinstein and Sarnak \cite{RuSa}, and its generalization by Akbary, Ng and Shahabi \cite{ANS}. 
\end{proof}

If in addition to RH we assume LI, then by Theorem 1.9 of Akbary, Ng and Shahabi \cite{ANS} we have the following explicit formula for the Fourier transform of $\mu_M$  
\begin{equation}\label{Fourier}
\widehat{\mu}_M(t)=\int_{-\infty}^{\infty} e^{-it} d\mu_M= e^{-it}\prod_{\gamma_n>0}J_0\left(\frac{2t}{\sqrt{\frac{1}{4}+\gamma_n^2}}\right),
\end{equation}
for all $t\in \mathbb{R}$, 
where $J_0(t)=\sum_{m=0}^{\infty}(-1)^m(t/2)^{2m}/m!^2$ is the Bessel function of order $0$.
We deduce

\begin{pro}\label{SumRand}
Assume RH and LI. Let $X(\gamma_n)$ be a sequence of independent random variables, indexed by the positive imaginary parts of the non-trivial zeros of $\zeta(s)$, and uniformly distributed on the unit circle. Then $\mu_M$ is the distribution of the random variable
$$Z= 1+ 2\re\sum_{\gamma_n>0}\frac{X(\gamma_n)}{\sqrt{\frac{1}{4}+\gamma_n^2}}.$$
\end{pro}
\begin{proof}  Note that $J_0(t)=\ex\Big(e^{-it \re X}\Big)$ where $X$ is a random variable uniformly distributed on the unit circle. Therefore,  since the $X(\gamma_n)$ are independent we obtain that 
$$ \ex\left(e^{-it Z}\right)=e^{-it}\prod_{\gamma_n>0}\ex\left(\exp\left(-i\frac{2t}{\sqrt{\frac{1}{4}+\gamma_n^2}}\re X(\gamma_n)\right)\right)=\widehat{\mu}_M(t).$$
Since the Fourier transform completely characterizes the distribution, we deduce that $\mu_M$ is the probability distribution of the random variable $Z$. 

\end{proof}

\begin{proof}[Proof of Theorem \ref{Bias}]
Since $Z$ is the sum of continuous random variables, then by Proposition \ref{SumRand} the probability distribution $\mu_M$ is absolutely continuous.  Let $\epsilon>0$ be given, and $f_1$ be a continuous function such that 
$$
f_1(x)=\begin{cases} 1 & \text{ if } x\geq 0 \\ \in [0,1] &\text{ if } -\epsilon <x< 0\\ 0 & \text{ if } x<-\epsilon.
\end{cases}
$$
Then it follows from Propositions \ref{LimitDistrib} and \ref{SumRand} that 
$$ \overline{\delta}(\mathcal{M})\leq \lim_{x\to\infty} \frac{1}{\log x}\int_2^xf_1\big(E_M(t)\big)\frac{dt}{t}=\int_{-\infty}^{\infty} f_1(t)d\mu_M \leq \mu_M(-\epsilon, \infty)=\pr(Z>0)+O(\epsilon),$$
since $\mu_M$  is absolutely continuous. 
Similarly, if $f_2$ is a continuous function such that 
$$
f_2(x)=\begin{cases} 1 & \text{ if } x\geq \epsilon \\ \in [0,1] &\text{ if } 0 < x< \epsilon  \\ 0 & \text{ if } x\leq 0.
\end{cases}
$$
Then 
$$ \underline{\delta}(\mathcal{M})\geq \lim_{x\to\infty} \frac{1}{\log x}\int_2^xf_2\big(E_M(t)\big)\frac{dt}{t}=\int_{-\infty}^{\infty} f_2(t)d\mu_M \geq \mu_M(\epsilon, \infty)=\pr(Z>0)+O(\epsilon).$$
Therefore, letting $\epsilon\to 0$ we deduce that 
\begin{equation}\label{Density}
\delta(\mathcal{M})=\pr(Z>0).
\end{equation}
Assuming RH and LI, Rubinstein and Sarnak \cite{RuSa} proved that the limiting logarithmic distribution of $(\pi(x)-\text{Li}(x))(\log x)/\sqrt{x}$ is the probability distribution of the random variable 
$$ \widetilde{Z}= -1+ 2\re\sum_{\gamma_n>0}\frac{X(\gamma_n)}{\sqrt{\frac{1}{4}+\gamma_n^2}}.$$
Since  the $X(\gamma_n)$ are symmetric random variables, it follows that $Z$ and $-\widetilde{Z}$ have the same distribution and hence
$$ P(Z>0)= P(\widetilde{Z}<0)= 1-P(\widetilde{Z}>0).$$
Finally, it follows from the same argument leading to \eqref{Density} that $P(\widetilde{Z}>0)$ is the logarithmic density of the set of real numbers $x\geq 2$ for which $\pi(x)>\text{Li}(x)$ and hence, from the computations of Rubinstein and Sarnak \cite{RuSa} we have $P(Z>0)=0.99999973...$

\end{proof}

In the remaining part of this section, we shall explain the heuristic behind Conjecture \ref{MaxMin}.  Note that 
$$
 \prod_{p\leq x}\left(1-\frac{1}{p}\right)^{-1}= e^{\gamma}(\log x) \exp\left(\frac{E_M(x)}{\sqrt{x}\log x}\right).
$$
Moreover, by Corollary \ref{ExplicitRH} and the Riemann-von  Mangoldt formula we have 
$$ E_M(x)\ll 1+ \sum_{0<\gamma_n< x}\frac{1}{\gamma_n}\ll (\log x)^2.$$
Therefore, we deduce that 
\begin{equation}\label{Approx3}
\begin{aligned}
\prod_{p\leq x}\left(1-\frac{1}{p}\right)^{-1}
&= e^{\gamma}(\log x) \left(1+ \frac{E_M(x)}{\sqrt{x}\log x}+ O\left(\frac{(\log x)^2}{x}\right)\right)\\
&= e^{\gamma}\log x+ e^{\gamma} \frac{E_M(x)}{\sqrt{x}}+ O\left(\frac{(\log x)^3}{x}\right).
\end{aligned}
\end{equation}
Furthermore, by the argument in the proof of Theorem \ref{Bias} we have, under RH and LI, that
\begin{equation}\label{Convergence}
\lim_{Y\to \infty}\frac1Y\text{meas}\left\{ 1\leq y\leq Y: E_M(e^y)>V\right\}=\pr(Z>V).
\end{equation}
Improving on a result of Montgomery \cite{Mont}, Monach \cite{Mona} showed that for $V\gg 1$ we have
$$ \pr(Z>V)= \exp\left(-C_0 \sqrt{V} \exp\left(\sqrt{2\pi V}\right)(1+o(1))\right),$$
for some explicit constant $C_0>0$. 
Therefore, if the convergence in \eqref{Convergence} is ``sufficiently uniform'' in $Y$, then one would deduce that 
$$ \sup_{1\leq y\leq Y} E_m(e^y)= \left(\frac{1}{2\pi}+o(1)\right)(\log\log y)^2,$$
and 
$$\inf_{1\leq y\leq Y} E_m(e^y)= \left(-\frac{1}{2\pi}+o(1)\right)(\log\log y)^2.$$
Inserting these estimates in \eqref{Approx3} yields Conjecture \ref{MaxMin}.

\end{document}